\documentclass[11pt]{article}
\usepackage[margin=1.1in]{geometry}
\usepackage{amsmath,amssymb,amsthm}
\usepackage{booktabs}
\usepackage{graphicx}
\usepackage{hyperref}
\newtheorem{theorem}{Theorem}
\newtheorem{proposition}{Proposition}
\title{Simple symmetric Venn diagrams with 17, 19 and 23 curves}
\author{Chris Dzoba}
\date{October 10, 2026}
\begin{document}
\maketitle

\begin{abstract}
We exhibit simple, rotationally symmetric Venn diagrams with 17 curves, with 19 curves and with 23 curves: $n$ Jordan
curves carried to one another by rotation through $2\pi/n$, with every one of the $2^n$ regions present and connected
and, since the diagrams are simple, every crossing on exactly two curves. Symmetric Venn diagrams exist for every prime
number of curves (Griggs, Killian and Savage, 2004), but those diagrams have many curves through a point; simple ones
were known only up to 13 curves (Mamakani and Ruskey, 2014). Four 17-curve, twelve 19-curve and six 23-curve diagrams
were found by a Metropolis walk on rotation-invariant quadrangulations of the sphere in which regions may temporarily
be duplicated, started from the Griggs--Killian--Savage diagram with its multiple crossings resolved. At 23 curves the
walk was held for two weeks by duplicated regions near the poles; the two lineages that finished were the first whose
$E=92$ states had none. Every diagram is given by a machine-checkable certificate; one certificate each of the 17-,
19- and 23-curve sizes has been verified by a formal proof in Lean~4. All of the diagrams are non-monotone, which is
why the crossing-sequence searches that found the 11- and 13-curve diagrams could not have found them.
\end{abstract}
\begin{figure}[p]
\centering
\includegraphics[width=0.92\textwidth]{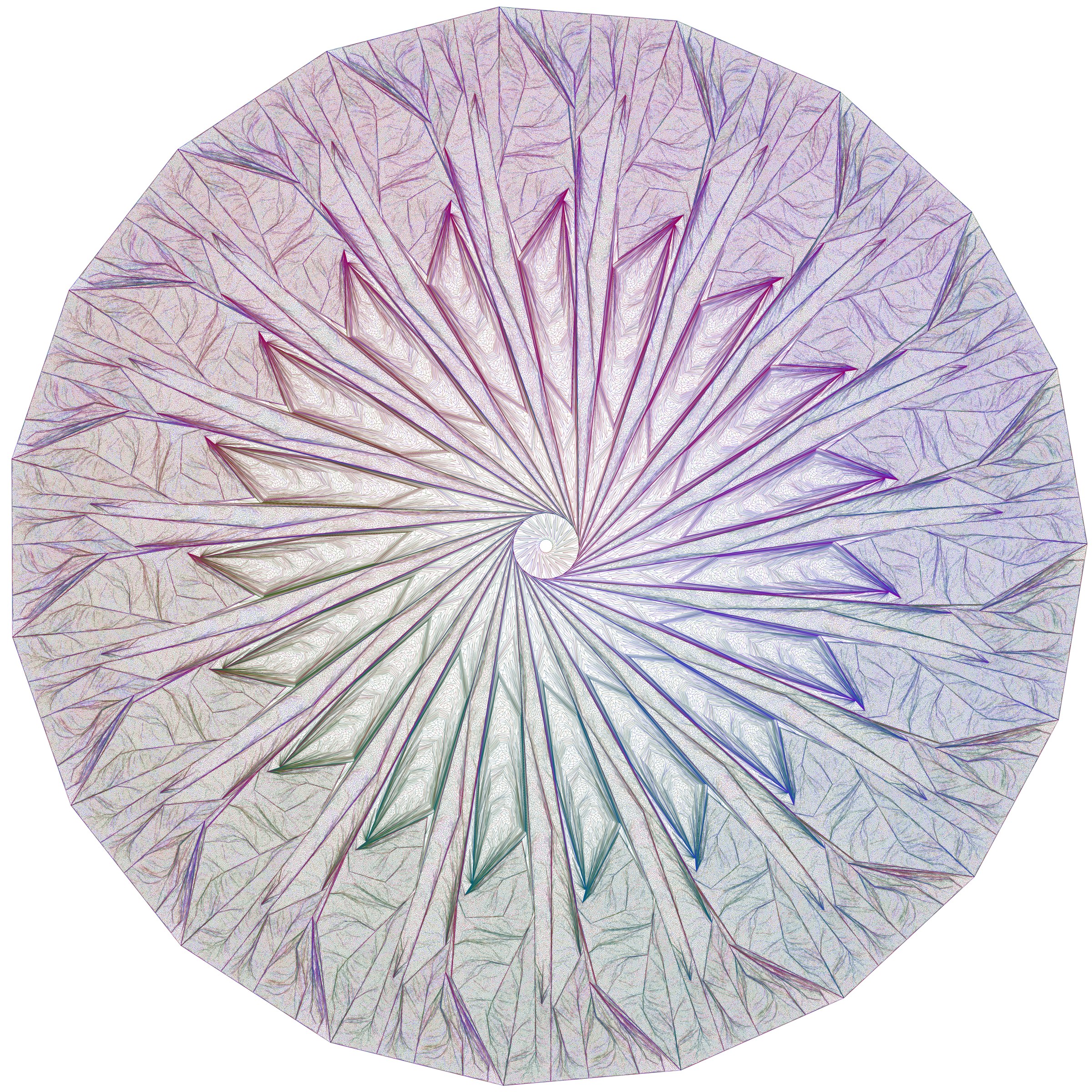}
\caption{One of the twelve 19-curve diagrams (certificate \texttt{fcb7ad34}), drawn from its certificate at near-uniform
crossing density by the plotter described in the data repository: 19 curves, $524{,}286$ crossings, every one of the
$524{,}288$ regions present exactly once. The centre hole is the region inside all 19 curves; the outer face is the
region outside all of them. Colours identify curves.}
\label{fig:venn19}
\end{figure}

\section{Statement}
An $n$-Venn diagram is a family of $n$ Jordan curves in the plane such that for every subset $S$ of the curves the
set of points inside exactly the curves of $S$ is nonempty and connected. It is \emph{simple} if no point lies on
three curves, and \emph{symmetric} if a rotation by $2\pi/n$ maps the family to itself. Henderson observed in 1963
that a symmetric $n$-Venn diagram can exist only for prime $n$ \cite{Hen63}; his argument had a gap, and the first
complete proof of his theorem is due to Wagon and Webb \cite{WW08}. Gr\"unbaum asked in 1975 whether they
exist for every prime and gave simple symmetric diagrams with five curves \cite{Gru75}; simple symmetric diagrams with
seven curves were found by Gr\"unbaum and by Edwards in the 1990s \cite{Gru92,Edw98}. Hamburger constructed the first
symmetric 11-Venn diagram in 2002 \cite{Ham02}; it is not simple, and the same holds for the diagrams of Griggs,
Killian and Savage, who then proved that symmetric $n$-Venn diagrams exist for every prime $n$ \cite{GKS04} (see also
\cite{KRSW04}). Simple symmetric diagrams with 11 and 13 curves were found by Mamakani and Ruskey in 2012 and 2014
\cite{MR12,MR14} by a search over monotone diagrams, and the question for larger primes, put forward as the central open
problem by Ruskey, Savage and Wagon \cite{RSW06}, remained open: the recent literature lists simple symmetric $n$-Venn
diagrams as known only for $n\le 13$ \cite{BGMV26}. The survey of Ruskey and Weston \cite{RW97} gives the history up to
its last revision in 2005, before the 11- and 13-curve diagrams; its list of open problems still asks for a simple
symmetric 11-Venn diagram.

\begin{theorem}
Simple symmetric Venn diagrams with 17 curves exist, and so do simple symmetric Venn diagrams with 19 curves and with
23 curves.
\end{theorem}

The proof is by exhibition. Four 17-curve, twelve 19-curve and six 23-curve diagrams are given as certificates in the
form of Section~\ref{sec:encoding}, each accepted by two independently written checkers, and one each of the 17-, 19-
and 23-curve certificates has been formally verified in Lean~4 (Section~\ref{sec:lean}). The 17-curve diagrams were
found on 17 September 2026 and the 19-curve diagrams on 20--23 September; the first nine 19-curve diagrams were
reported in the first version of this note. The 23-curve diagrams were found on 5 and 6 October 2026 and are described
in Section~\ref{sec:23}; the second version of this note (6 October) described the first five of them. This version
adds the last three 19-curve diagrams, the sixth 23-curve diagram, the formal proof of a 23-curve certificate and a
measurement of the depth of the regions of one 23-curve diagram (Section~\ref{sec:properties}).

\section{Encoding and verification}\label{sec:encoding}
Let $D$ be a simple symmetric $n$-Venn diagram, drawn on the sphere with the unbounded region as a face. Its dual
map $M$ has one vertex per region, labelled by the region's membership pattern in $\{0,1\}^n$; one edge per arc,
joining labels that differ in exactly the bit of the arc's curve; and one face per crossing, a 4-cycle with labels
$u,\,u{\oplus}e_a,\,u{\oplus}e_a{\oplus}e_b,\,u{\oplus}e_b$ (a \emph{cube square}). Conversely:

\begin{proposition}
A map on the sphere whose vertices are labelled by distinct elements of $\{0,1\}^n$, all $2^n$ occurring, whose
faces are cube squares, whose edges of each coordinate $i$ form a single cycle through the faces containing them,
and whose vertex rotations are single cycles, is the dual of a simple $n$-Venn diagram; if the label rotation
$x\mapsto(x_1,\dots,x_{n-1},x_0)$ is an orientation-preserving automorphism of the map, the diagram is symmetric.
\end{proposition}

\begin{proof}[Proof sketch]
The map is a cellular embedding on the sphere; its geometric dual is a 4-regular plane graph whose vertices are the
faces (crossings). The edges of coordinate $i$, chained through opposite sides of each crossing, form one closed curve
$C_i$; regions of the arrangement are the vertices of $M$, and by the label condition each point of $\{0,1\}^n$ is
realized by exactly one region, which is a disk. Rotation of the labels acts as a rotation of the sphere fixing the
two vertices $0^n$ and $1^n$; taking $0^n$ as the outer face gives a plane diagram with $n$-fold rotational symmetry.
\end{proof}

A certificate is the list of the $2^n-2$ faces of $M$, each as four $n$-bit labels: $131{,}070$ faces for $n=17$,
$524{,}286$ for $n=19$ and $8{,}388{,}606$ for $n=23$ (889~MB as JSON). Verification recomputes: all $2^n$ labels
present exactly once, every edge in exactly two faces ($2^{n+1}-4$ edges), Euler characteristic 2, every face
transversal, a single rotation cycle at every vertex, one cycle per curve with both sides connected, and invariance
under label rotation. This takes seconds for 17 and 19 curves; on the six 23-curve certificates the checker in
the data repository took 693, 734, 610, 633, 614 and 706 seconds with peak memory of 20.0, 21.4, 21.4, 22.5, 21.4 and
21.4~GB. Two independently written
checkers, a half-edge implementation in C++ and a face-set implementation in Python, accept all twenty-two certificates. A
third, written later and sharing no code with either, audits the raw search state directly; it has been run on the first nine
19-curve certificates and on one of the four 17-curve certificates (\texttt{venn17-gcp-s16}, \texttt{431857aa}), whose
raw state it reports as a valid instance embedding with 17 curves, $131{,}072$ vertex instances and Euler
characteristic 2. On the raw states of the first five 23-curve diagrams, Codex's native inspector, a program pinned by its hash that
reimports the raw state, rebuilds the full graph and checks the vertex, edge and face counts, the Euler characteristic
and the rotation orbits, with the input's SHA-256 checked before and after each run, reported order 23, energy 0 and
no missing or duplicated label on 6 October (c25 at 06:16, c19, c21 and the pilot diagram at 06:17--06:19, f4 at
08:15; the c25 receipt was reassessed after an error in the wrapper's orbit formula, which had assumed that every
vertex orbit has 23 members, with no native rerun). The third checker has not been run on the last three 19-curve
certificates, nor the native inspector on the sixth 23-curve state.

\begin{table}[h]
\centering\small
\begin{tabular}{llll}
\toprule
$n$ & certificate & SHA-256 prefix & non-isomorphic \\
\midrule
17 & \texttt{venn17-local-c3-s2} & \texttt{c178d7bd} & yes \\
17 & \texttt{venn17-gcp-s12} & \texttt{87f981b8} & yes \\
17 & \texttt{venn17-gcp-s14} & \texttt{eec8cd55} & yes \\
17 & \texttt{venn17-gcp-s16} & \texttt{431857aa} & yes \\
19 & \texttt{venn19-closure-s196002} & \texttt{fcb7ad34} & yes \\
19 & \texttt{venn19-closure-s196004} & \texttt{285de8ea} & yes \\
19 & \texttt{venn19-closure-s195001} & \texttt{a09ceae0} & yes \\
19 & \texttt{venn19-closure-s196001} & \texttt{56e7e1a1} & yes \\
19 & \texttt{venn19-closure-s196007} & \texttt{312e8c88} & yes \\
19 & \texttt{venn19-closure-s195002} & \texttt{726d1fb0} & yes \\
19 & \texttt{venn19-fresh-s192015} & \texttt{45a8aee9} & yes \\
19 & \texttt{venn19-fresh-s190002} & \texttt{b46bf5d0} & yes \\
19 & \texttt{venn19-fresh-s192007} & \texttt{d5b20c27} & yes \\
19 & \texttt{venn19-ramp12h-s192104} & \texttt{0a4afd95} & yes \\
19 & \texttt{venn19-ramp12h-s192102} & \texttt{de295445} & yes \\
19 & \texttt{venn19-ramp12h-s192106} & \texttt{b2fec494} & yes \\
23 & \texttt{venn23-c25-s230025} & \texttt{adc5a02e} & yes \\
23 & \texttt{venn23-c19-s230019} & \texttt{5dfddac2} & yes \\
23 & \texttt{venn23-c21-s230021} & \texttt{e6849a0b} & yes \\
23 & \texttt{venn23-c6p-s1065101} & \texttt{6e923fbc} & yes \\
23 & \texttt{venn23-f4-s254004} & \texttt{5888e16a} & yes \\
23 & \texttt{venn23-f6-s254006} & \texttt{310b9571} & yes \\
\bottomrule
\end{tabular}
\caption{The twenty-two certificates. Within each size the diagrams are pairwise non-isomorphic under cyclic
relabelling of the curves, mirror image and exchange of the two poles, by canonical forms compared over all $4n$ such
maps; the first five 23-curve certificates are moreover pairwise non-isomorphic under arbitrary relabelling of the curves,
by their region-degree histograms, which differ for every pair (the histogram of the sixth has not been computed). As
face sets the first nine 19-curve certificates differ pairwise by between $25{,}460$ and $947{,}188$ faces
of $524{,}286$; each of the three fresh-start certificates among them shares only about a tenth of its faces with any
other. The 23-curve certificates \texttt{adc5a02e} and \texttt{5dfddac2} share $827{,}149$ of their
$8{,}388{,}606$ faces ($9.9\%$) in the frame in which they were found, and the same under every cyclic relabelling,
since each face set is invariant under label rotation; the canonical forms of \texttt{adc5a02e}, \texttt{5dfddac2},
\texttt{e6849a0b}, \texttt{6e923fbc}, \texttt{5888e16a} and \texttt{310b9571} over all $4n=92$ maps (\texttt{400e5279},
\texttt{337076ae}, \texttt{56f659b4}, \texttt{029acb28}, \texttt{2670bfda} and \texttt{0345e49f}) differ pairwise. The
prefixes of the six 19-curve closure rows are those of the raw search
states from which the certificates were exported (the JSON form of \texttt{fcb7ad34} is \texttt{ed26b3ba}); all other
prefixes, including those of the six fresh-start 19-curve rows and of the 23-curve rows, are those of the certificate
JSON files, the form in which they are published, and the JSON prefixes of all twenty-two are listed in
\texttt{certificates/SHA256SUMS}.}
\end{table}

\section{Formal verification}\label{sec:lean}
The first 17-curve certificate (\texttt{c178d7bd}) was formally verified in Lean~4 with Mathlib by Justin Grimes \cite{Gri26},
working from the definitions alone and without access to our code. His theorem states that the 17 curves
constructed from the certificate are Jordan curves in the Euclidean plane, carried to one another by a rigid
rotation through $2\pi/17$, with every one of the $2^{17}$ regions path-connected, no point on three curves, and every
crossing a transversal crossing of two curves. The finite parts of the proof use Lean's compiled evaluation
(\texttt{native\_decide}); the development builds from scratch in under an hour.

The first 19-curve certificate (\texttt{fcb7ad34}; JSON form \texttt{ed26b3ba}) was verified by a port of the same
development to $n=19$, carried out by Codex (OpenAI) from Grimes's sources: the dimension-dependent constants were
changed, the finite meridian and sector certificates that the proof consumes were regenerated for 19, and the
vendored topology library was left untouched. The port builds in twenty minutes and proves
\texttt{exists\_simple\_rotational\_venn\_19}: there exist $C:\mathrm{Fin}\,19\to\mathcal P(\mathbb R^2)$ and sides $S$
with \texttt{SimpleRotationalVenn}~$C$~$S$, the same predicate as at 17. Its axiom list is that of the 17-curve proof
under the namespace rename: \texttt{propext}, \texttt{Classical.choice}, \texttt{Quot.sound} and the twenty
\texttt{native\_decide} certificates. We re-ran \texttt{\#print axioms} on the built project to confirm this.

The 23-curve certificate c25 (\texttt{adc5a02e}, the first 23-curve certificate to be verified) was verified by a
further port of the same development to $n=23$, again carried out by Codex, which proves
\texttt{exists\_simple\_rotational\_venn\_23} (namespace \texttt{Venn23.Topology}), the same \texttt{SimpleRotationalVenn} statement as at 17
and 19, for the diagram built from that certificate. The proof completed at 11:03:58 EDT on 6 October 2026: the build
of the proof (\texttt{lake build Venn23 VennTopology}, 3{,}471 jobs) and the final axiom audit both returned 0. The theorem
depends on \texttt{propext}, \texttt{Classical.choice}, \texttt{Quot.sound} and fifteen \texttt{native\_decide}
computations, eighteen axioms in all, with no \texttt{sorry}; as at 17 and 19, Lean's compiler and runtime are trusted
for the \texttt{native\_decide} computations. Two changes from the 19-curve port were needed at this size, neither of
which changes a statement or a checker definition: Lean's default heartbeat limit was removed on the two certificate
theorems (the meridian and sector certificates) where elaboration had reached it, and the seven sector conditions are
proved together by one \texttt{native\_decide} over their conjunction, from which the existing named theorems follow by
projection, so the proof has fifteen \texttt{native\_decide} dependencies where the 19-curve proof has twenty. Four runs
were needed. The first passed the port's compiled finite checker (all of the finite combinatorial checks:
$8{,}388{,}608$ regions, $16{,}777{,}212$ arcs, $8{,}388{,}606$ crossings) and generated the meridian and sector
witnesses that the proof consumes, then stopped at its 16~GiB memory limit; the second and third stopped at the
heartbeat limit, in the meridian and the sector certificate respectively; the fourth, with the changes above and a
64~GiB limit, passed. The 23-curve development is not in the git repository; it is published as the release asset and
Zenodo file \texttt{venn23-lean-proof.zip}. The other five 23-curve certificates have not been formalized.

\section{Method}\label{sec:method}
We search the space of rotation-invariant cube-square quadrangulations of the sphere in which several vertices may
carry the same label. The energy of a state is $E = (\text{missing labels}) + \lambda\,(\text{duplicate regions})$;
$E=0$ with $\lambda>0$ is exactly a certificate. Moves are the local Reidemeister-type rewirings of the dual (lens
insertion and removal on a 3-cube, triangle flip, bigon insertion and removal), each applied simultaneously to all $n$
rotation images; every move preserves the sphere, the cube-square faces and the one-cycle-per-curve property, so the
walk never leaves the space of valid symmetric arrangements. Proposals are accepted by the Metropolis rule at
temperature $T$, with half of the proposals targeted at a missing label or a duplicated region.

Allowing duplicates is what makes the search work: with duplicates forbidden, greedy growth of the same
quadrangulations reaches $98.3\%$ of the labels and jams, because the last lens insertions require two missing labels
to coincide geometrically. With duplicates allowed, a missing label can be created next to a duplicate at zero cost and
defects diffuse until they annihilate.

\paragraph{17 curves.} The walk starts from the Griggs--Killian--Savage diagram \cite{GKS04} for $n=17$ with every multiple
crossing resolved into simple crossings, a state with all labels present and $69{,}394$ duplicate regions. With
$\lambda$ ramped from $0.3$ to $1$ over three hours at $T=0.25$, three of fifteen independent runs reached $E=0$; a
fourth diagram was obtained by a further 40-minute cycle from a state with $68$ duplicates. Ten later instrumented
runs of the same recipe finished seven times within three hours, first reaching $E=2n$ after $1.6$--$2.8\times10^9$
proposals and finishing within $1.8\times10^9$ more. The same method finds 11- and 13-curve diagrams from the trivial
starting state in seconds and minutes: on a laptop busy with other work, 200 of 200 seeds produced a simple symmetric
11-Venn diagram, in a median of 0.8 seconds and at most 16 seconds, all 200 pairwise non-isomorphic (up to cyclic
relabelling, mirror image and swapping the poles) and all non-monotone; 11 of 13 seeds produced a 13-curve diagram within
a ten-minute limit, in a median of 52 seconds, again all distinct and all non-monotone.

\paragraph{19 curves: budget.} The resolved starting diagram for $n=19$ has all labels and $305{,}862$ duplicate
regions. Scaled by the number of regions, no 19-curve run in the first three days received the budget that the
17-curve runs needed: the longest had $6\times10^9$ proposals and was still descending when stopped, and every restart
from a low-energy state was shorter than the 17-curve endgame itself.

\paragraph{19 curves: the knot.} Every lineage that reached low energy stalled with the same duplicated label,
$\{0,3,4\}$, whose two instances had degrees 6--9 and shared a neighbour. A region of the dual can move by a triangle
flip only at degree 3, so the walk's targeted proposals, when they landed on such a duplicate, did nothing. The
duplicate is inherited: in the resolved starting diagram the label $\{0,3,4\}$ already appears twice, once as a region
of degree 18. (The resolved 17-curve diagram has the same feature for $\{0,1,4\}$, degrees 4 and 17, and the
17-curve walk happens to grind it down; the ten 17-curve pre-completion states all have their duplicates at degree 3
or 4.) One added proposal class removes the obstruction: when a targeted proposal lands on a duplicate of degree
above 3, it attacks that degree, by a triangle flip at a degree-3 neighbour or a lens insertion on an incident edge,
both ordinary moves under the ordinary acceptance rule. From a frozen $E=38$ state, 25 minutes at $\lambda=0.35$ with
this proposal reduced $\{0,3,4\}$ to multiplicity one; without it, 45 minutes at $\lambda=0.30$ left it untouched.
A three-hour ramp $\lambda: 0.5\to1$ then cooled the state to $E=38$ with two missing labels of rank 11 and no
duplicates.

\paragraph{19 curves: closure.} From that state, continuations differing only in the random seed were run with
$\lambda$ held at 1 and with $\lambda$ ramped from $0.6$, $0.8$ or $0.9$ to 1. Held at 1, five of eight
continuations reached $E=0$, after 4 seconds, 13 seconds, 190 seconds, 4.3 hours and 6.5 hours; of the sixteen
ramped continuations one finished, at the end of its six-hour ramp. The pre-completion states are saved exactly, by
journal undo of the winning move inside the engine, and show three closing mechanisms: a single triangle flip of the
last duplicated region, of degree 3, into the last missing label at Hamming distance exactly 3 (four diagrams, and one of
the 36-hour fresh-start diagrams below); a
single lens insertion creating both missing labels at once after they had drifted to Hamming distance 1 (one
diagram, and two of the 36-hour fresh-start diagrams below); and a single lens removal of two adjacent duplicated regions in a state with no missing label (one diagram).

\paragraph{19 curves: fresh start.} The six diagrams above descend from one hand-off state. The other six do not.
Thirty-two 36-hour runs were started from the resolved scaffold on 20 September: twenty with the ordinary proposals
(twelve on the fleet and eight on a laptop), six with the degree-attacking proposal applied to three in ten of the
targeted landings at $T=0.25$, and six with it applied always at $T=0.28$, all with $\lambda$ ramped linearly from
$0.3$ to $1$ over the wall time. Three reached $E=0$: two of the three-in-ten runs, after 32.6 and 33.3 hours, and one
laptop run with the ordinary proposals, after 34.0 hours, at $1.1\times10^{10}$ to $2.1\times10^{10}$ proposals. The
first closed by the triangle flip above, from a state with one missing label of rank 8 and one duplicated label of
rank 11, both copies of degree 3, at Hamming distance 3; the other two closed by the lens insertion, from states with no
duplicate and two missing labels at Hamming distance 1. Each of the three shares only about a tenth of its faces with
any other of the first nine diagrams in the table. When the September version of this note was written the remaining three-in-ten runs
stood at best energies 38, 38, 114 and 152, the ordinary runs between 38 and 228, and the always-on runs no lower than 114.

\paragraph{19 curves: shorter ramp.} Six further fresh starts on the laptop, from 22 September, used the three-in-ten
recipe at $T=0.25$ with $\lambda$ ramped from $0.3$ to $1$ over 12 hours instead of 36 and then held at $1$. None
closed during the ramp, which ended at best energies between 38 and 190. Three closed during the hold: after 52 minutes
of it (13.0 hours in all), after 86 minutes (13.4 hours) and after 11.6 hours (23.6 hours), the last on 23 September.
The first two took about $5{,}500$ and $5{,}800$ proposals per vertex in all, against $8{,}700$ to $16{,}000$ for the
three closures on the 36-hour ramp: closures came once $\lambda$ was near 1, and the longer ramp mostly waited for
it. Their pre-completion states, saved by journal undo as above, have two missing labels and no
duplicate (\texttt{s192104} and \texttt{s192106}) and one missing and one duplicated label (\texttt{s192102}); the
closing moves themselves were not classified. The twelve 19-curve diagrams are pairwise non-isomorphic by their
canonical forms over all $4n=76$ maps.

\begin{figure}[p]
\centering
\includegraphics[width=0.95\textwidth]{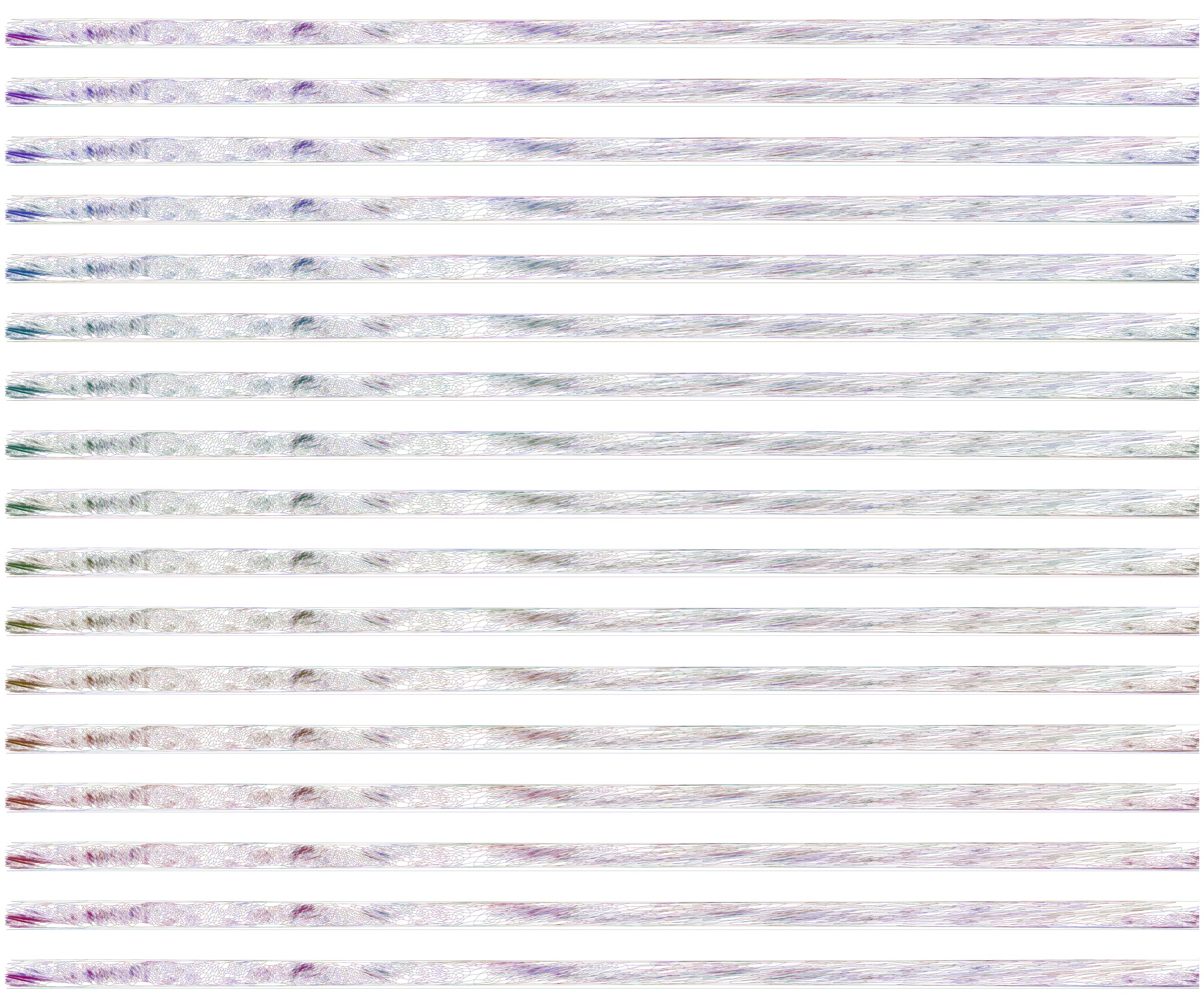}
\caption{The 17-curve diagram \texttt{c178d7bd} unrolled: its seventeen congruent sectors, each a fundamental domain
of the rotation containing $131{,}070/17 = 7{,}710$ crossings, drawn as rows, with the colour of each curve
advancing by one position from row to row. The whole diagram is one row bent around the centre seventeen times.}
\label{fig:sector}
\end{figure}

\section{23 curves}\label{sec:23}
Times in this section are Eastern Daylight Time (UTC$-4$); the dated run records (the launch manifest and the cloud
log), which are kept by the author, carry UTC as well.

\paragraph{Scale and engine.} The resolved Griggs--Killian--Savage diagram for $n=23$ has all $8{,}388{,}608$ labels
and $5{,}323{,}948$ duplicate regions, seventeen times the 19-curve count, and a full state occupies 1.4~GB on disk. A
second starting state, the \emph{orbit-safe} scaffold, was prepared by Codex from the same diagram: of the 54 duplicated
labels of low or high rank that had an instance of degree 8 or more (the 23-curve counterparts of the
inherited $\{0,3,4\}$ at 19, which is one of them), 33 were removed before the walk began, the total number of
duplicate regions being unchanged. The walk itself was
rewritten by Codex to store one representative of each rotation orbit and apply every move through the rotation, with
the same moves, energy, acceptance rule and export format; on a fixed workload it ran 3.3 times faster than the
full-state engine at 23 curves in 1.1~GB of memory. A control at 17 curves on 29 September, eight seeds of the
original engine against eight of the rewrite on the three-hour recipe of Section~\ref{sec:method}, closed 5 of 8 and
8 of 8 within the wall, so the difficulties below are not an artefact of the rewrite. In the walk's own units the
energy at $\lambda=1$ is always a multiple of $2n=46$, as it was a multiple of $38$ at 19, and $E=46$ is the last
rung before a certificate.

\paragraph{The crust.} The 23-curve campaign ran from 20 September to 6 October on one laptop and a succession of
cloud instances. For two weeks every lineage that reached low energy stopped at $E=230$, $138$, $92$ or $46$ and stayed. A
census of the duplicated labels in every stalled state we examined showed the same pattern: at least one duplicated
label in the \emph{crust}, rank at most 4 or at least $n-4$. The clearest case is the lineage c5, an eight-day
$\lambda$ ramp from the orbit-safe scaffold followed by a 24-hour cooling from $T=0.25$ to $0.14$: it reached $E=46$ on
1 October with the label $\{0,7,9\}$ duplicated at degrees 4 and 4, where the triangle flip cannot reach it, and
eleven continuations on the cloud held that state for three days without improvement; a melting test (twelve chains,
the degree-attacking proposal at lower $\lambda$ for three different durations, then a ramp back) left the label at
two copies in 28 of 28 saved states. Its sibling c8, from the standard scaffold, sat at
$E=92$ with $\{0,3,4\}$ and three labels of rank 20--21 duplicated, and 21 continuations for three days never reached
$E=46$. Twenty-four fresh coolings from c5's end-of-ramp state, where $\{0,7,9\}$ was still single, avoided that knot in
nine of the thirteen that reached $E=46$, but those nine carried crust pairs instead ($\{0,2,5,6\}$ with $\{0,8\}$, or
$\{0,3,4\}$ with $\{0,2,5,6\}$), and none went below 46 in the 6--13 hours before that instance's budget ended.
The two lineages that did finish, c6 (orbit-safe scaffold) and c9 (standard scaffold), had the longest ramps of any
lineage that reached low energy, and were the first whose $E=92$ states had no crust duplicate at all.

\paragraph{The recipe.} Both lineages were launched on 22 September with the degree-attacking proposal applied to
three in ten of the targeted landings and $\lambda$ ramped linearly from $0.3$ at $T=0.25$; the ramp was re-based that
evening, about twenty hours after launch, to reach $\lambda=1$ after twelve days, which it did on 4 October at engine time $1{,}101{,}666$ seconds (c6)
and $1{,}101{,}665$ seconds (c9), with $E=3{,}082$ (c6) and $3{,}036$ (c9). After about an hour's hold at $\lambda=1$ and $T=0.25$ ($3{,}499$ seconds for
c6, $3{,}373$ for c9), a 24-hour cooling from $T=0.25$ to $0.14$ at $\lambda=1$ followed, from 22:21 on 4 October.
Sixteen to seventeen hours into the cooling, on 5 October, both crossed $E=92$ (c9 at 14:16, c6 at 15:07) with two duplicated labels of
rank 8--12 and two missing labels of rank 9--16, and no crust defect. Those states were saved, and continuations were
started from them, and from the $E=46$ states they soon produced, with $T$ held at $0.14$ and $\lambda$ at 1: six on
the laptop (three at 19:02, two at 20:45, one at 21:12) and three on a cloud instance; three more (two on the laptop,
one on the cloud) ran the 19-curve closure recipe, $T=0.25$ with the degree-attacking proposal applied always. Of these,
c25, c19 and c21 on the laptop and f4 and f6 on the cloud produced diagrams reported here, and the pilot diagram,
described below, came from a run that was not among these continuations. c25, from c6's $E=92$ state, reached $E=46$ after
81~minutes (two duplicated labels, no missing), passed through the form with one duplicate and one missing label, rose
to $E=92$ again, and reached $E=0$ at engine time $9{,}985$ seconds, at 00:00:09 on 6 October, 2~h~47~min~41~s after launch.
c19, from a state of c9 with one missing and three duplicated labels, spent six hours between $E=92$ and $184$; its
$E=46$ state was exported at 01:16, about $22{,}400$ seconds after launch (the log's progress lines are $1{,}800$ seconds
apart and first report $E=46$ at engine time $23{,}413$ seconds), and it reached $E=0$ at engine time $25{,}892$
seconds, at 02:15 on 6 October, 7~h~12~min~33~s after launch.
c21 descends from c6 by a different route: c20, launched at 19:02 from c6's $E=92$ state, reached $E=46$ sixteen
minutes later in the form with one duplicated and one missing label, and c21 was launched at 20:45 from that state.
Its log shows it moving between the two forms of $E=46$ (one duplicate and one missing label; two duplicated labels and
none missing) with excursions to $E=92$, and it reached $E=0$ at engine time $29{,}528$ seconds, at 04:58 on 6 October,
8~h~13~min~18~s after launch.

The fourth diagram to be verified was found before any of these, by a run that was not among the continuations above
and whose result was read only after the three above had been found and verified; we call it the pilot diagram. On
the evening of 5 October a paired pilot
of Codex's memory-assisted engine (see ``What did not work'' below) was launched on the laptop from an
earlier state of c6, the $E=184$ export of its cooling (SHA-256 prefix \texttt{87b75d5b}): two arms in sequence,
first an ordinary Metropolis walk with the engine's memory features inactive (its run record shows no tabu rejection,
no backtrack and no aspiration commit), then a backtracking arm, each with a wall of $11{,}400$ seconds, at
$T=0.14$ and $\lambda=1$, seed 1065101, launched at 21:37. The ordinary arm reached $E=46$ after $2{,}065$
search seconds, in the form with two missing labels and no duplicate, and $E=0$ after $4{,}158$ search seconds, at
about 22:46 on 5 October; the run wrapper's native validation passed and the run stopped at 22:49, so the
backtracking arm never ran. Its $E=0$ state (\texttt{6e569819}; the certificate exported from it by the original
engine is \texttt{6e923fbc}) passed the same two checkers as the other three on the morning of 6 October.

The fifth diagram to be verified came from the cloud instance. Its lane f4 was a hold at $T=0.14$ and $\lambda=1$ on a four-vCPU
instance (\texttt{c2d-highcpu-4}, region \texttt{us-east4-b}), launched at about 21:23 on 5 October, seed 254004, from
c6's own $E=46$ state of 19:41 on 5 October (SHA-256 prefix \texttt{8e900bf9}: one duplicated label of rank 11, with
copies of degrees 3 and 5, and one missing label). At the 01:21 readout of the instance two of its three holds, f4
among them, stood at $E=46$ and the third at $E=92$, and its lane on the 19-curve recipe had not improved; f4's log then
shows the same wandering as c21's between the two forms of $E=46$, with excursions to $E=92$ and once to $E=138$, and
it reached $E=0$ at engine time $34{,}272$ seconds, at 06:54 on 6 October, 9~h~31~min after launch, from the insertion
form (two missing labels, of ranks 9 and 10, one bit apart, and no duplicate). Its zero-step reload in the original
engine and the independent Python checker both passed at 07:35 on 6 October; the certificate exported from it is
\texttt{5888e16a}, and its canonical form (\texttt{2670bfda}) differs from those of the four diagrams above, so
although it descends from lineage c6, as do three of the four, it is a fifth diagram, and the first found on a cloud
instance rather than the laptop.

The sixth diagram came from the same instance. Its lane f6 was a hold at $T=0.14$ and $\lambda=1$, seed 254006,
launched at about the same time as f4 from c6's $E=92$ state, the state from which c25 had started (SHA-256 prefix
\texttt{2937cd1f}). It reached $E=0$ at engine time $45{,}207$ seconds, at 09:56 on 6 October, 12~h~33~min after
launch, from the insertion form (two missing labels, of ranks 10 and 11, one bit apart, and no duplicate). Its
zero-step reload in the original engine and the independent Python checker both passed at 10:38 on 6 October; the
certificate exported from it is \texttt{310b9571}, and its canonical form (\texttt{0345e49f}) differs from those of
the five diagrams above, so f6 and c25, which began from the same state with different seeds, are different diagrams.
In order of discovery the six diagrams are therefore the pilot diagram (about 22:46 on 5 October), c25 (00:00:09), c19
(02:15), c21 (04:58), f4 (06:54) and f6 (09:56 on 6 October); all six descend from the two twelve-day ramps, five from
c6 (the pilot diagram from its $E=184$ state, c25 and f6 from its $E=92$ state, c21 from c20's $E=46$ state and f4
from its own $E=46$ state) and one, c19, from c9.
When the second version of this note was prepared (6 October) the other three laptop continuations stood at $E=46$ and
the two on the laptop running the 19-curve closure recipe had not improved on their starting energies.

\paragraph{The closing move.} The pre-completion states, saved exactly as at 19, have two forms. For c25, c19, f4
and f6 it is $E=46$ with no duplicated region and two missing labels, of ranks 9 and 10 (10 and 11 for f6), at Hamming
distance 1 (for c19, after
a cyclic relabelling; the two labels are 9 bits apart in the frame of the file): the form from which three of the
19-curve diagrams closed by a single lens insertion creating both missing labels at once. For c21 it is the opposite
form, $E=46$ with no missing label and two duplicated labels, both of rank 10, with copies of degrees 3 and 4 and of
degrees 2 and 3, two bits apart in the frame of the file (no cyclic relabelling brings them closer): the form from which one 19-curve diagram
(\texttt{726d1fb0}) closed by a single lens removal, though there the two duplicated labels were adjacent after a cyclic
relabelling. So c21 closed by a deletion, removing both duplicated orbits at once, the first closure of that type at 23.
The energies and label counts in the run logs are consistent with these moves; unlike at 19, the winning moves
themselves were not recovered from the journal, so the mechanisms are inferred from the pre-completion states rather
than recorded. The pilot diagram's run did not export a pre-completion state and did not record its closing move;
its telemetry records the all-missing form ($E=46$, two missing labels, no duplicate) at its first passage of $E=46$
and again at its last periodic sample, 18 seconds before closure, so a state of that form preceded the closure, but
whether the finishing move was an insertion is not established. A breadth-first search over count-neutral triangle
flips from the campaign's $E=46$ states, run while
they were stalled, had found no finishing move to depth 6; the closures of c25, c19, f4 and f6 came from the other side of
that family, by insertion from states without duplicates, and c21's from a state of the family's own form, but one that
the search had never been run on.

\paragraph{What did not work.} In order, with the numbers of runs: a cloud fleet of 56 branches from the lowest
laptop states (26--30 September) on which flat cooling at $T=0.10$ and $0.08$ froze at $E$ 500--600, standard cooling
sat at 230--276 for five hours, a melt at $\lambda=0.35$ with the degree-attacking proposal blew 16 branches up to
$E\approx22{,}000$, and the 138 floor held for 21 of the 56 branches with none lower in 35 hours; three laptop holds at
$E=138$ (c3 from 28 September and its restart c11 from 30 September, both stopped on 1 October, and c4, stopped on
3 October after about three days at 138); the eight-day ramps c5 and c8, with the 11 and 21 cloud continuations and
12 melting chains above; the 24 re-coolings from c5's end of ramp; two batches of four three-hour arms each of a
memory-assisted (tabu and backtracking) variant of the engine written by Codex, in which memory off beat backtracking in
both pairs of the first batch (best $E$ 1,426 and 1,518 against 1,840 and 1,748) and no arm of the second went below
$E=92$ (the later paired pilot of the same engine from c6's $E=184$ state is the run that found the pilot diagram,
with its memory features inactive, so that run is evidence neither for nor against them); the exact flip search
above (no finisher); six 600-second quenches at $T=0.01$ from a stalled $E=46$ state,
all of which stayed at 46; three continuations on the 19-curve closure recipe; two gated fresh starts stopped at
$E\approx780$ and $550$; and 112 fresh starts on two 56-vCPU spot instances from 29 September, on a four-day ramp,
one-day hold, one-day cooling and a hold at $T=0.14$, half from each scaffold, none of which had gone below $E=138$
when the second version of this note was prepared.

\paragraph{Verification and distinctness.} Each of the six states was checked as at 19: a zero-step reload in the
original full-state engine with its structure check (missing 0, duplicates 0, energy 0), then the independent Python
checker on the exported faces ($V=8{,}388{,}608$, $E=16{,}777{,}212$, $F=8{,}388{,}606$, Euler characteristic 2, every
edge in two faces, every face a transversal cube square, one rotation cycle per vertex, both sides of every curve
connected, invariant under label rotation), and then the public checker of the data repository on the certificate JSON
(all six PASS, in 693, 734, 610, 633, 614 and 706 seconds). All six are non-monotone
(Section~\ref{sec:properties}).
c25 and c19 share
$827{,}149$ of their $8{,}388{,}606$ faces, about a tenth, the same proportion as between the fresh-start diagrams at
19 (the face-set overlaps of the other four certificates were not computed). Distinctness was established in two ways. First, by the method used for the 17- and 19-curve
certificates: the canonical form of each certificate, the minimum over all $4n=92$ maps (cyclic relabelling, mirror
image, exchange of the poles) of the hash of its sorted relabelled face list, was computed for c25, c19, c21, the
pilot diagram, f4 and f6 (\texttt{400e5279}, \texttt{337076ae}, \texttt{56f659b4}, \texttt{029acb28},
\texttt{2670bfda} and \texttt{0345e49f}), and the six forms differ pairwise: six files, six isomorphism classes. Second,
the region-degree histograms of the first five certificates, computed independently by Codex from the certificate files, differ
for every one of the ten pairs (for example $2{,}951{,}107$, $2{,}950{,}003$, $2{,}937{,}192$, $2{,}947{,}197$ and
$2{,}944{,}046$ regions of degree 3 in c25, c19, c21, the pilot diagram and f4); the degree histogram is unchanged by
any relabelling of the curves, not only the 92 maps, so those five diagrams are pairwise non-isomorphic under arbitrary
relabelling (Table~1), although c25, c21, the pilot diagram and f4 all descend from lineage c6. The histogram of f6 has
not been computed, so f6 is separated from the other five by the canonical form alone.
% FIGURE (23 curves): a pressure-layout drawing of venn23-c25-s230025 is not yet available (as of 6 October only
% Tutte layouts and the pressure-layout log existed); when figures/venn23-c25.jpg exists, insert here:
% \begin{figure}[p]\centering\includegraphics[width=0.92\textwidth]{figures/venn23-c25.jpg}
% \caption{The 23-curve diagram c25 (certificate \texttt{adc5a02e}): 23 curves, $8{,}388{,}606$ crossings, every one
% of the $8{,}388{,}608$ regions present exactly once.}\label{fig:venn23}\end{figure}

\section{Properties}\label{sec:properties}
All twenty-two diagrams are non-monotone in the sense of Bultena, Gr\"unbaum and Ruskey \cite{BGR99}: in the first 19-curve certificate, $3{,}590$ of the $27{,}594$ region orbits
($13\%$) lack a neighbouring region of rank one lower or one higher, the 17-curve certificates behave the same
way, and in the 23-curve certificates c25, c19, c21, \texttt{6e923fbc}, f4 and f6 $52{,}742$, $52{,}279$, $52{,}597$,
$52{,}518$, $52{,}501$ and $52{,}727$ of the $364{,}722$ region orbits ($14\%$) lack such a neighbour (the two poles, fixed by the rotation, are
monotone and are not counted among the orbits here or in the 19-curve figure). Since monotone diagrams are exactly those
that can be drawn with convex curves \cite{BGR99}, none of them
has a convex drawing. Every previously published simple symmetric Venn diagram with $n\ge 11$ is
monotone, and the crossing-sequence representation used to find them \cite{MR14,MMR12} enumerates monotone diagrams by
construction, so the diagrams reported here were outside the search space of those methods; Mamakani and Ruskey
themselves noted that the size of their search sequences grows about tenfold from 13 to 17 curves \cite{MR14}. Minimum region degree is 3 (a degree-2 region is impossible in a simple
Venn diagram with $n\ge3$); in the first 19-curve certificate the degree histogram is $187{,}340$ regions of degree 3,
$199{,}120$ of degree 4, $98{,}686$ of degree 5, and a tail to degree 16, with the two poles of degree 19 and every
rank-1 and rank-18 region of degree 11.

Most regions of c25 are many steps from the outer region. A breadth-first search by Codex in the dual map $M$ of
c25, in which two regions are adjacent when they share an arc ($16{,}777{,}212$ adjacencies), starting from the outer
region $0^{23}$, reaches all $8{,}388{,}608$ regions, and every region is within 51 steps of the outer region. The mean
depth is $22.11$, the median 22, the 90th percentile 31 and the 99th percentile 37, and $42.34\%$ of the regions
have depth greater than 23; the region inside all curves has depth exactly 23, and the 115 deepest regions, at depth 51,
have ranks 9, 11 and 13. This has a consequence for drawings with regions of equal area. Suppose c25 is drawn so that
its $8{,}388{,}607$ bounded regions have disjoint interiors and equal area $a$, and their union is a disk of radius $r$.
A point of a region at depth $d$ is joined to the outer boundary by $d$ segments, each with both ends in the closure of
a single region, so it lies within $d$ maximal region diameters of the boundary; hence $r\le 51\,\delta$, where $\delta$
is the largest region diameter, and since $\pi r^2=8{,}388{,}607\,a$, some region has diameter
$\delta\ge\sqrt{8{,}388{,}607/\pi}\,\sqrt a/51\approx 32.04\sqrt a$ (or $31.94\sqrt a$ if the union is a regular
23-gon and $r$ its inradius), against $1.13\sqrt a$ for a round region of the same area. This is a necessary condition
on such a drawing under those hypotheses, not a proof that equal-area drawings do not exist, and it says nothing about
drawings whose regions fill some other domain.

\section{Data}
Certificates, checker transcripts, the search code with its checksums, the independent checkers and the
17- and 19-curve Lean developments are at \url{https://github.com/dzoba/venn17} and archived on Zenodo under the concept DOI
\url{https://doi.org/10.5281/zenodo.22885649}, which resolves to the latest version. The six 23-curve certificates are
889~MB each and are not in the git repository: they are carried as gzip-compressed assets (about 103~MB each) of the
repository's release \texttt{v1.3} and in the Zenodo record of version 1.3,
\url{https://doi.org/10.5281/zenodo.23189412}, together with the 23-curve Lean development
(\texttt{venn23-lean-proof.zip}), with the SHA-256 sums of both the \texttt{.json} and the \texttt{.json.gz} files in the repository's README
and of the \texttt{.json} files in \texttt{certificates/SHA256SUMS}. For all three sizes the published certificates,
the checker transcripts and the checksums are the public record; the pre-completion states (of the 17-, 19- and
23-curve diagrams alike), the run logs and the dated run records of the 23-curve campaign (the launch manifest and
the cloud log) are kept by the author.

\section*{Tool and computational resource disclosure}
This work was carried out by two AI systems working under the author's direction, over five days for the 17- and
19-curve diagrams and a further sixteen days, from 20 September to 6 October 2026, for the 23-curve diagrams, and this
note discloses their roles in full, in the spirit of the Leiden Declaration on Artificial Intelligence and Mathematics.
Claude (Anthropic) formulated the dual-map encoding, designed and implemented the search (the relaxed walk, its moves,
energy and annealing schedules), diagnosed the failure at 19 from the run records, designed and implemented the added
proposal class and the capture of pre-completion states, ran the searches and the verification, performed the analyses
reported here and drafted this note; at 23 it also ran the campaign day to day (the laptop lineages and the cloud
fleets, their schedules, the defect censuses that identified the crust obstruction, and the choice of the states to
continue from), keeping a dated log of every launch and stop. Codex (OpenAI) constructed the resolved
Griggs--Killian--Savage starting states and the orbit-safe 23-curve scaffold, wrote the quotient-native rewrite of the
search engine used at 23, wrote independent validators for the 17- and 19-curve certificates, identified the need for
the bigon moves, built exact local solvers used to diagnose the plateaus at 19 and at 23, built the memory-assisted engine
variants tested at 23 and ran the pilot of that engine whose ordinary arm found the first of the 23-curve diagrams in
order of discovery, ported
the Lean development to 19 and to 23 curves, and computed the dual-map depths of Section~\ref{sec:properties}. The author set the goals and strategy,
decided what to pursue and what to abandon, approved the launches and stops recorded in the campaign log, supplied the computing
resources, and is responsible for the content. Computing: one 16-core laptop (Apple M4~Max, 128~GB) throughout, on
which four of the six 23-curve diagrams were found, the other two (f4 and f6) on a four-vCPU cloud instance; for the 19-curve
runs cloud instances of up to 24 vCPUs each, never more than
about a hundred vCPUs at once and none for longer than two days, at a total cloud cost under five hundred dollars; for
the 23-curve campaign spot and standard instances of up to 64 vCPUs each, at most four at once (about 170 vCPUs), the
longest-lived running from 29 September and still running when the second version of this note was prepared.
Software: the search engine, checkers and plotter in the data repository (C++ and Python with NumPy); Lean~4 with
Mathlib for the formal proofs; \texttt{tectonic} for this document.

\section*{Acknowledgments}
We thank Justin Grimes for the formal verification of the 17-curve certificate in Lean, carried out independently and
without knowledge of our construction, on which the 19- and 23-curve verifications are built. We thank Khalegh Mamakani for his
encouragement from the first day, for pointing out that the earlier searches excluded non-monotone diagrams by
construction, and for his advice on how this work should be presented; and Carla Savage for her kind reply and,
with Jerrold Griggs and Charles Killian, for the construction from which every one of these searches
begins. The survey of Ruskey and Weston was our map of the field throughout.

\end{document}